\documentclass[11pt,letterpaper]{article}
\usepackage{amsfonts}
\usepackage{amsmath}
\usepackage{amssymb}
\providecommand{\U}[1]{\protect\rule{.1in}{.1in}}
\newtheorem{theorem}{Theorem}
\newtheorem{corollary}[theorem]{Corollary}
\newtheorem{lemma}[theorem]{Lemma}
\newtheorem{proposition}[theorem]{Proposition}
\newtheorem{remark}[theorem]{Remark}
\newenvironment{proof}[1][Proof]{\noindent\textbf{#1.} }{\ \rule{0.5em}{0.5em}}

\begin{document}

\title{Pathwise uniqueness and continuous dependence for SDEs with nonregular drift}
\author{E. Fedrizzi$^{1}$, F. Flandoli$^{2}$\\ \small{(1) Laboratoire de Probabilit\'es et Mod\`eles Al\'eatoires, Universit\'{e} Paris VII, France}\\
\small{(2) Dipartimento di Matematica Applicata, Universit\`{a} di Pisa, Italia }}
\maketitle

\begin{abstract}
A new proof of a pathwise uniqueness result of Krylov and R\"{o}ckner is given. It concerns SDEs with drift having only certain integrability properties. In spite of the poor regularity of the drift, pathwise continuous dependence on initial conditions may be obtained, by means of this new proof. The proof is formulated in such a way to show that the only major tool is a good regularity theory for the heat equation forced by a function with the same regularity of the drift.

\end{abstract}

\section{Introduction}

Consider the stochastic differential equation in $\mathbb{R}^{d}$
\begin{equation}
X_{t}=x+\int_{0}^{t}b\left(  s,X_{s}\right)  \mathrm{d}s+W_{t},\qquad
t\in\left[  0,T\right]  \label{SDE}
\end{equation}
where $W$ is a $d$-dimensional Brownian motion on a filtered probability space
$\left(  \Omega,F_{t},P\right)  $, $x\in\mathbb{R}^{d}$ and $b:\left[
0,T\right]  \times\mathbb{R}^{d}\rightarrow\mathbb{R}^{d}$ is a measurable
vector field with components of class $L_{p}^{q}\left(  T\right)
:=L^{q}\left(  0,T;L^{p}\left(  \mathbb{R}^{d}\right)  \right)  $ for some
$p,q\in\left(  1,\infty\right)  $ satisfying the condition
\begin{equation}
\frac{d}{p}+\frac{2}{q}<1 \label{pq}
\end{equation}
(known in fluid dynamics, with $\leq$, as the Prodi-Serrin condition). A
measurable function $f:\left[  0,T\right]  \times\mathbb{R}^{d}\rightarrow
\mathbb{R}$ is in $L_{p}^{q}\left(  T\right)  $ if
\[
\Vert f\Vert_{L_{p}^{q}\left(  T\right)  }:=\left(  \int_{0}^{T}\left(
\int_{\mathbb{R}^{d}}\left\vert f(r,y)\right\vert ^{p}\,\mathrm{d}y\right)
^{q/p}\mathrm{d}r\right)  ^{1/q}<\infty.
\]
A remarkable result of Krylov and R\"{o}ckner \cite{KR05}, which elaborates
previous results of many authors, including Zwonkin \cite{Zv74}, Veretennikov
\cite{Ver}, Portenko \cite{Port}, states that this equation has a unique
strong solution, in the class of continuous processes such that
\begin{equation}
P\left(  \int_{0}^{T}\left\vert b\left(  s,X_{s}\right)  \right\vert
^{2}\mathrm{d}s<\infty\right)  =1. \label{condition on solutions}%
\end{equation}
They also remark that the solution depends continuously on $x$ in probability.
The result is extended in \cite{KR05} to local $L_{p}^{q}$-integrability
conditions plus growth conditions;\ and there are extensions to
state-dependent diffusion coefficients and other regularity assumptions, see
\cite{Zhang} and references therein.

The aim of this note is to give a new proof of the same result, based on a
different argument, essentially based only on regularity theory of the heat
equation with forcing equal to the drift or of the same class of regularity.
We hope this new proof will look more elementary. The new proof is somewhat
more quantitative (see in particular proposition \ref{proposition estimate})
and will allow us to show the $\alpha$-H\"{o}lder continuous dependence on
$x$, for every $\alpha<1$, \textit{pathwise}, in the spirit of stochastic
flows. This result is new and somewhat surprising, being $b$ so rough.
Precisely, we prove:

\begin{theorem}
\label{theorem}Equation (\ref{SDE}), with $b\in L_{p}^{q}\left(  T\right)  $
with $p,q\in\left(  1,\infty\right)  $ satisfying the condition (\ref{pq}),
for every $x\in\mathbb{R}^{d}$ has a unique strong solution $X_{t}^{x}$ such
that (\ref{condition on solutions}) holds true. The random field $\left\{
X_{t}^{x},t\in\left[  0,T\right]  ,x\in\mathbb{R}^{d}\right\}  $ has a
continuous modification, $\alpha$-H\"{o}lder continuous in $x$, for every
$\alpha<1$.
\end{theorem}

As we said, the aim of this note is to show a new simple argument to deal with
SDEs with nonregular drift. In this spirit, we prefer to keep the exposition as
simple as possible and thus we limit ourselves to the two claims of the
theorem (uniqueness and pathwise H\"{o}lder continuity in the initial
conditions). However, with longer arguments, we have also checked that an
$\alpha$-H\"{o}lder continuous stochastic flow exists; and moreover the
solution is differentiable in $x$ in an average sense, but not pathwise (we
cannot get a differentiable stochastic flow). These results, mostly included
in \cite{Fedrizzi}, will be published elsewhere. Moreover, we do not stress
the generality beyond the (already challenging) class $L_{p}^{q}\left(
T\right)  $, but it is clear that one can accept some form of local
integrability plus suitable control of the growth, at the expenses of several
more details. And presumably the extension to other regularity classes
different from $L_{p}^{q}\left(  T\right)  $ is possible, preserving at least
the basic property that $\nabla u$ is bounded (see below).

It will be clear from the proof below that a sort of principle emerges.
\textit{If we have a good theory for the heat equation}%
\begin{equation}
\frac{\partial u}{\partial t}+\frac{1}{2}\Delta u=\varphi\text{ on }\left[
0,T\right]  ,\qquad u|_{t=T}=0 \label{heat eq}%
\end{equation}
\textit{when }$\varphi$\textit{ has the same regularity as the drift }$b$,
\textit{then we have the main tools to prove strong uniqueness} and possibly
stochastic flows of H\"{o}lder maps. The good theory must include (at our
present level of understanding) a \textit{uniform bound on the gradient}
$\nabla u$. This is the main reason for the assumption $b\in L_{p}^{q}\left(
T\right)  $ with $p,q\in\left(  1,\infty\right)  $ satisfying condition
(\ref{pq}). Other properties of $u$, of course, are used below but they look
more flexible, not optimized. It seems that this principle extends to infinite
dimensional situations (replacing the heat semigroup by the Ornstein-Uhlenbeck
one) and other finite dimensional cases beyond the one treated here.

Of course, this principle is just a reformulation of a known fact, because the
non-trivial results on Kolmogorov type equations needed in other proofs of
pathwise uniqueness (like those in the references mentioned above, see also
\cite{FGP}, \cite{DaPratoFlandoli}), are ultimately based on a perturbative
analysis of the heat equation, in spaces with regularity related to the one of
the drift. See also remark \ref{remark on equivalence}. But the presentation
here is very direct and easily prompt to generalizations.

The proof, indeed, becomes slightly shorter if we use a good regularity theory
for the backward Kolmogorov equation%
\[
\frac{\partial U}{\partial t}+\frac{1}{2}\Delta U+\left(  b\cdot\nabla\right)
U=-b\text{ on }\left[  0,T\right]  ,\qquad U_{\Phi}|_{t=T}=0.
\]
This is the approach developed in \cite{FGP} for H\"{o}lder continuous drift
(see also the infinite dimensional generalization \cite{DaPratoFlandoli}), and
in \cite{Fedrizzi} for $L_{p}^{q}$-drift. The proof is shorter (and to some
extent more far reaching, if one wants to prove further properties like
differentiability in $x$), but at the price of a careful preliminary analysis
of the Kolmogorov equation. Even if ultimately the two approaches are
equivalent, we think it is conceptually interesting to realize that only heat
equation estimates, with forcing of the same type as the drift, are needed.
For this reason we give a self-contained proof based only on (\ref{heat eq}).

\section{First step of the proof}

First, let us clarify that we prove only the strong uniqueness and pathwise
dependence part of the theorem. Indeed, we give for granted the weak existence
proved in previous works by means of Girsanov theorem (see \cite{KR05} and
proposition \ref{teo weak existence} in the appendix) and thus the strong
existence follows from weak existence and strong uniqueness by the classical
Yamada-Watanabe theorem, or by the construction given by Gyongy and Krylov
\cite{Gyongy-Krylov}.

Consider the backward heat equation (\ref{heat eq}) with $\varphi\in L_{p}%
^{q}\left(  T\right)  $. Denote by $H_{2,p}^{q}\left(  T\right)  $ the space%
\[
H_{2,p}^{q}\left(  T\right)  :=L^{q}\left(  0,T;W^{2,p}\left(  \mathbb{R}%
^{d}\right)  \right)  \cap W^{1,q}\left(  0,T;L^{p}\left(  \mathbb{R}%
^{d}\right)  \right)
\]
with norm $\left\Vert .\right\Vert _{H_{2,p}^{q}\left(  T\right)  }$ given by
the sum of the natural norms of $L^{q}\left(  0,T;W^{2,p}\left(
\mathbb{R}^{d}\right)  \right)  $ and $W^{1,q}\left(  0,T;L^{p}\left(
\mathbb{R}^{d}\right)  \right)  $. Denote by $\left\Vert .\right\Vert
_{L^{\infty}\left(  T\right)  }$ the norm in the space $L^{\infty}\left(
\left[  0,T\right]  \times\mathbb{R}^{d}\right)  $. All our analysis will be
based only on the following classical result (see Krylov \cite{Kr01a} and
\cite[ lemma 10.2]{KR05}). More is known (uniqueness, H\"{o}lder continuity of
$\nabla u$), but we insist that we use only the following properties.

\begin{theorem}
For every $\varphi\in L_{p}^{q}\left(  T\right)  $, the backward heat equation
(\ref{heat eq}) has at least a solution $u\in H_{2,p}^{q}\left(  T\right)  $,
with%
\begin{equation}
\left\Vert D^{2}u\right\Vert _{L_{p}^{q}\left(  T\right)  }\leq C\left\Vert
\varphi\right\Vert _{L_{p}^{q}\left(  T\right)  }. \label{maximal}%
\end{equation}
Moreover, $\nabla u\in L^{\infty}\left(  \left[  0,T\right]  \times
\mathbb{R}^{d}\right)  $ and%
\begin{equation}
\left\Vert \nabla u\right\Vert _{L^{\infty}\left(  T\right)  }\leq C\left(
T\right)  \left\Vert \varphi\right\Vert _{L_{p}^{q}\left(  T\right)  }
\label{bounded gradient}%
\end{equation}
with%
\begin{equation}
\lim_{T\rightarrow0}C\left(  T\right)  =0. \label{small C}%
\end{equation}

\end{theorem}

Given a vector field $\Phi:\left[  0,T\right]  \times\mathbb{R}^{d}%
\rightarrow\mathbb{R}^{d}$, we still write $\Phi\in L_{p}^{q}\left(  T\right)
$ when all components $\Phi^{i}$ are of class $L_{p}^{q}\left(  T\right)  $.
Denote by $U_{\Phi}$ the $\mathbb{R}^{d}$-valued field such that $U_{\Phi}%
^{i}$ solves the heat equation above with $\varphi=-\Phi^{i}$;\ in vector
notations%
\[
\frac{\partial U_{\Phi}}{\partial t}+\frac{1}{2}\Delta U_{\Phi}=-\Phi\text{ on
}\left[  0,T\right]  ,\qquad U_{\Phi}|_{t=T}=0.
\]
We have $U_{\Phi}^{i}\in H_{2,p}^{q}(T)$, $i=1,...,d$. As above, we write
$U_{\Phi}\in H_{2,p}^{q}(T)$, for simplicity of notations.

Moreover, denote by $\mathcal{T}:L_{p}^{q}\left(  T\right)  \rightarrow
L_{p}^{q}\left(  T\right)  $ the map defined as%
\[
\mathcal{T}\left(  \Phi\right)  :=\left(  b\cdot\nabla\right)  U_{\Phi}.
\]

Using a generalization of It\^{o} formula to $H_{2,p}^{q}(T)$-functions (see
\cite[theorem 3.7]{KR05}), if $X$ is a solution to equation (\ref{SDE}) we
have%
\[
dU_{\Phi}\left(  t,X_{t}\right)  =-\Phi\left(  t,X_{t}\right)  \mathrm{d}%
t+\left(  b\left(  t,X_{t}\right)  \cdot\nabla\right)  U_{\Phi}\left(
t,X_{t}\right)  +\nabla U_{\Phi}\left(  t,X_{t}\right)  \mathrm{d}W_{t}%
\]
namely%
\[
\int_{0}^{t}\Phi\left(  s,X_{s}\right)  \mathrm{d}s=U_{\Phi}\left(
0,x\right)  -U_{\Phi}\left(  t,X_{t}\right)  +\int_{0}^{t}\mathcal{T}\left(
\Phi\right)  \left(  s,X_{s}\right)  \mathrm{d}s+\int_{0}^{t}\nabla U_{\Phi
}\left(  s,X_{s}\right)  \mathrm{d}W_{s}.
\]
Hence, taking $\Phi=b$, we can rewrite equation (\ref{SDE}) in the form%
\[
X_{t}=x+U_{b}\left(  0,x\right)  -U_{b}\left(  t,X_{t}\right)  +\int_{0}%
^{t}\mathcal{T}\left(  b\right)  \left(  s,X_{s}\right)  \mathrm{d}s+\int
_{0}^{t}\nabla U_{b}\left(  s,X_{s}\right)  \mathrm{d}W_{s}+W_{t}.
\]

Let us make several comments on this reformulation of equation (\ref{SDE}).
The difficulty in (\ref{SDE}) is the non-regular field $b$, only of class
$L_{p}^{q}\left(  T\right)  $. The terms $U_{b}\left(  0,x\right)  $,
$U_{b}\left(  t,X_{t}\right)  $ and $\int_{0}^{t}\nabla U_{b}\left(
s,X_{s}\right)  \mathrm{d}W_{s}$ of the new equation involve more regular
fields: $U_{b}$ has even H\"{o}lder continuous gradient, while $\nabla U_{b}$
has gradient in $L_{p}^{q}\left(  T\right)  $. On the contrary, the term
$\int_{0}^{t}\mathcal{T}\left(  b\right)  \left(  s,X_{s}\right)  \mathrm{d}s$
is not better than then original one, $\int_{0}^{t}b\left(  s,X_{s}\right)
\mathrm{d}s$, from the regularity viewpoint. But, if we take small $T$, the
$L_{p}^{q}\left(  T\right)  $-norm of $\left(  b\cdot\nabla\right)  U_{b}$ is
small as we want (because of (\ref{small C})). So we have replaced the
non-regular term in (\ref{SDE}) by more regular ones plus a term which has the
same degree of regularity but is much smaller. Iterating this procedure,
namely replacing $\int_{0}^{t}\mathcal{T}\left(  b\right)  \left(
s,X_{s}\right)  \mathrm{d}s$ by analogous terms, and so on $n$ times, we may
keep the time interval $\left[  0,T\right]  $ small but given, and decrease
arbitrarily the size of the non-regular term. We shall see that the sum of the
other term is under control.

To be more precise, we repeat what we have done above for $\int_{0}%
^{t}b\left(  s,X_{s}\right)  \mathrm{d}s$ and get
\[
\int_{0}^{t}\hspace{-0,1cm}\mathcal{T}\left(  b\right)  \left(  s,X_{s}\right)
\mathrm{d}s=U_{\mathcal{T}\left(  b\right)  }\left(  0,x\right)
-U_{\mathcal{T}\left(  b\right)  }\left(  t,X_{t}\right)  +\int_{0}%
^{t}\hspace{-0,1cm}\mathcal{T}^{2}\left(  b\right)  \left(  s,X_{s}\right)  \mathrm{d}%
s+\int_{0}^{t}\hspace{-0,1cm}\nabla U_{\mathcal{T}\left(  b\right)  }\left(  s,X_{s}\right)
\mathrm{d}W_{s}.
\]
We iterate this procedure, substitute in the original equation and get
\begin{align*}
X_{t}  &  =x+\sum_{k=0}^{n}U_{\mathcal{T}^{k}\left(  b\right)  }\left(
0,x\right)  -\sum_{k=0}^{n}U_{\mathcal{T}^{k}\left(  b\right)  }\left(
t,X_{t}\right)  +\int_{0}^{t}\mathcal{T}^{n+1}\left(  b\right)  \left(
s,X_{s}\right)  \mathrm{d}s\\
&  +\int_{0}^{t}\left(  \sum_{k=0}^{n}\nabla U_{\mathcal{T}^{k}\left(
b\right)  }\left(  s,X_{s}\right)  \right)  \mathrm{d}W_{s}+W_{t}.
\end{align*}
where we have set $\mathcal{T}^{0}\left(  b\right)  =b$. We shall prove our
results (uniqueness and pathwise continuous dependence on initial conditions)
for this equation.

To simplify a little the notations, let us set%
\[
U^{\left(  n\right)  }\left(  t,x\right)  =\sum_{k=0}^{n}U_{\mathcal{T}%
^{k}\left(  b\right)  }\left(  t,x\right)  ,\qquad b^{\left(  n\right)
}=\mathcal{T}^{n+1}\left(  b\right)  .
\]
The equation reads%
\[
X_{t}=x+U^{\left(  n\right)  }\left(  0,x\right)  -U^{\left(  n\right)
}\left(  t,X_{t}\right)  +\int_{0}^{t}b^{\left(  n\right)  }\left(
s,X_{s}\right)  \mathrm{d}s+\int_{0}^{t}\left(  \nabla U^{\left(  n\right)
}\left(  s,X_{s}\right)  +I\right)  \cdot\mathrm{d}W_{s}.
\]

We discuss first the case when $b$ is H\"{o}lder continuous, both to see this
equation at work in an easier case, and to show two different ways to handle
such an equation, in the $C_{b}^{\alpha}$ and $L_{p}^{q}$ cases.

\begin{remark}
\label{remark on equivalence}Intuitively speaking (it can be made rigorous),
if we pass to the limit in the previous identity we get%
\[
X_{t}=x+U\left(  0,x\right)  -U\left(  t,X_{t}\right)  +\int_{0}^{t}\left(
\nabla U\left(  s,X_{s}\right)  +I\right)  \cdot\mathrm{d}W_{s}%
\]
where $U$ is the solution of the backward Kolmogorov equation%
\[
\frac{\partial U}{\partial t}+\frac{1}{2}\Delta U+\left(  b\cdot\nabla\right)
U=-b\text{ on }\left[  0,T\right]  ,\qquad U_{\Phi}|_{t=T}=0
\]
used in \cite{FGP} (for H\"{o}lder continuous drift). These two approaches are
thus equivalent, in principle, but for conceptual reasons and possibly for
future extensions we would like to give a proof explicitly based only on the
heat equation.
\end{remark}

\section{The case when $b$ is H\"{o}lder continuous}

For $\alpha\in\left(  0,1\right)  $, denote by $C_{b}^{\alpha}\left(
\mathbb{R}^{d}\right)  $ the space of all continuous $f:\mathbb{R}%
^{d}\rightarrow\mathbb{R}$ such that
\[
\left\Vert u\right\Vert _{C_{b}^{\alpha}\left(  T\right)  }:=\sup
_{x\in\mathbb{R}^{d}}\left\vert u\left(  x\right)  \right\vert +\sup_{x\neq
y}\frac{\left\vert u\left(  x\right)  -u\left(  y\right)  \right\vert
}{\left\vert x-y\right\vert ^{\alpha}}<\infty.
\]
In this case we use the following well known result. In fact maximal
regularity $u\in C\left(  \left[  0,T\right]  ;C_{b}^{2,\alpha}\left(
\mathbb{R}^{d}\right)  \right)  $ is known, and uniqueness, but again we do
not need it for our result and strategy of proof.

\begin{theorem}
For all $\varphi\in C\left(  \left[  0,T\right]  ;C_{b}^{\alpha}\left(
\mathbb{R}^{d}\right)  \right)  $ there exists at least one solution $u$ to
the backward heat equation (\ref{heat eq}) of class
\[
u\in C\left(  \left[  0,T\right]  ;C_{b}^{2}\left(  \mathbb{R}^{d}\right)
\right)  \cap C^{1}\left(  \left[  0,T\right]  ;C_{b}^{\alpha}\left(
\mathbb{R}^{d}\right)  \right)
\]
with%
\begin{equation}
\left\Vert D^{2}u\right\Vert _{L^{\infty}\left(  T\right)  }\leq C\left\Vert
\varphi\right\Vert _{C_{b}^{\alpha}\left(  T\right)  }
\label{second derivative estimate}%
\end{equation}
and
\begin{equation}
\left\Vert \nabla u\right\Vert _{C_{b}^{\alpha}\left(  T\right)  }\leq
C\left(  T\right)  \left\Vert \varphi\right\Vert _{C_{b}^{\alpha}\left(
T\right)  }\text{ with }\lim_{T\rightarrow0}C\left(  T\right)  =0.
\label{bounded gradient 2}%
\end{equation}

\end{theorem}

Due to estimate (\ref{second derivative estimate}), the proof of theorem
\ref{theorem} simplifies a lot. Let us remark that in this case theorem
\ref{theorem} is known, see \cite{FGP}, where it is proved that the equation
has a stochastic flow of diffeomorphisms.

\begin{lemma}
\label{lemma 1}Set%
\[
C_{n}\left(  T\right)  :=\sum_{k=0}^{n}\left\Vert \nabla U_{\mathcal{T}%
^{k}\left(  b\right)  }\right\Vert _{L^{\infty}\left(  T\right)  },\qquad
D_{n}\left(  T\right)  :=\sum_{k=0}^{n}\left\Vert D^{2}U_{\mathcal{T}%
^{k}\left(  b\right)  }\right\Vert _{L^{\infty}\left(  T\right)  }%
\]
Then there exists $T_{0}$ and $\ C>0$ such that for all $T\in(0, T_{0}]$ we have%
\[
C_{n}\left(  T\right)  \leq\frac{1}{2},\qquad D_{n}\left(  T\right)  \leq
C,\qquad\left\Vert \mathcal{T}^{n}\left(  b\right)  \right\Vert _{C_{b}%
^{\alpha}\left(  T\right)  }\leq C\frac{1}{2^{n}}%
\]
for every $n\in\mathbb{N}$.
\end{lemma}

\begin{proof}
Let
\[
\varepsilon=\left(  4\left\Vert b\right\Vert _{C_{b}^{\alpha}\left(  T\right)
}\right)  ^{-1}%
\]
(unless $b=0$, when there is nothing to prove). Due to
(\ref{bounded gradient 2}), we may choose $T_{0}$ such that for all $T\in
(0, T_{0}]$ we have%
\[
\left\Vert \nabla U_{\Phi}\right\Vert _{C_{b}^{\alpha}\left(  T\right)  }%
\leq\varepsilon\left\Vert \Phi\right\Vert _{C_{b}^{\alpha}\left(  T\right)
}.
\]
Thus%
\[
\left\Vert \nabla U_{b}\right\Vert _{C_{b}^{\alpha}\left(  T\right)  }%
\leq\varepsilon\left\Vert b\right\Vert _{C_{b}^{\alpha}\left(  T\right)  }%
\]%
\[
\left\Vert \mathcal{T}^{1}\left(  b\right)  \right\Vert _{C_{b}^{\alpha
}\left(  T\right)  }\leq\left\Vert \nabla U_{b}\right\Vert _{C_{b}^{\alpha
}\left(  T\right)  }\left\Vert b\right\Vert _{C_{b}^{\alpha}\left(  T\right)
}\leq\varepsilon\left\Vert b\right\Vert _{C_{b}^{\alpha}\left(  T\right)
}^{2}%
\]%
\[
\left\Vert \nabla U_{\mathcal{T}^{1}\left(  b\right)  }\right\Vert
_{C_{b}^{\alpha}\left(  T\right)  }\leq\varepsilon^{2}\left\Vert b\right\Vert
_{C_{b}^{\alpha}\left(  T\right)  }^{2}%
\]%
\[
\left\Vert \mathcal{T}^{2}\left(  b\right)  \right\Vert _{C_{b}^{\alpha
}\left(  T\right)  }\leq\left\Vert \nabla U_{\mathcal{T}^{1}\left(  b\right)
}\right\Vert _{C_{b}^{\alpha}\left(  T\right)  }\left\Vert b\right\Vert
_{C_{b}^{\alpha}\left(  T\right)  }\leq\varepsilon^{2}\left\Vert b\right\Vert
_{C_{b}^{\alpha}\left(  T\right)  }^{3}%
\]
and so on; by induction, one can see that%
\begin{align*}
\left\Vert \nabla U_{\mathcal{T}^{k}\left(  b\right)  }\right\Vert
_{C_{b}^{\alpha}\left(  T\right)  }  &  \leq\varepsilon^{k+1}\left\Vert
b\right\Vert _{C_{b}^{\alpha}\left(  T\right)  }^{k+1}\leq4^{-\left(
k+1\right)  }\\
\left\Vert \mathcal{T}^{k}\left(  b\right)  \right\Vert _{C_{b}^{\alpha
}\left(  T\right)  }  &  \leq\varepsilon^{k}\left\Vert b\right\Vert
_{C_{b}^{\alpha}\left(  T\right)  }^{k+1}\leq4^{-k}\left\Vert b\right\Vert
_{C_{b}^{\alpha}\left(  T\right)  }.
\end{align*}
Thus $C_{n}\left(  T\right)  \leq\sum_{k=0}^{n}4^{-\left(  k+1\right)  }%
\leq1/2$ and $\left\Vert \mathcal{T}^{n}\left(  b\right)  \right\Vert
_{L_{p}^{q}\left(  T\right)  }\leq C\frac{1}{2^{n}}$ for some constant $C>0$.
Moreover, for some constants $C^{\prime},C^{\prime\prime}>0$
\[
D_{n}\left(  T\right)  \leq\sum_{k=0}^{n}C^{\prime}\left\Vert \mathcal{T}%
^{k}\left(  b\right)  \right\Vert _{C_{b}^{\alpha}\left(  T\right)  }\leq
C^{\prime\prime}\sum_{k=0}^{n}2^{-k}\leq2C^{\prime\prime}.
\]
The proof is complete.
\end{proof}

Let $X_{t}^{\left(  i\right)  }$, $i=1,2$, be two solutions, with initial
conditions $x^{\left(  i\right)  }$, $i=1,2$. Given any $p\geq2$, let us
estimate
\begin{equation}
E\left[  \sup_{t\in\left[  0,T\right]  }\left\vert X_{t}^{\left(  1\right)
}-X_{t}^{\left(  2\right)  }\right\vert ^{p}\right]  . \label{to be estimated}%
\end{equation}
We have%
\[
X_{t}^{\left(  1\right)  }-X_{t}^{\left(  2\right)  }=a_{t}+b_{t}+c_{t}+d_{t}%
\]
where%
\begin{align*}
a_{t}  &  =x^{\left(  1\right)  }-x^{\left(  2\right)  }+U^{\left(  n\right)
}\left(  0,x^{\left(  1\right)  }\right)  -U^{\left(  n\right)  }\left(
0,x^{\left(  2\right)  }\right) \\
b_{t}  &  =U^{\left(  n\right)  }\left(  t,X_{t}^{\left(  1\right)  }\right)
-U^{\left(  n\right)  }\left(  0,X_{t}^{\left(  2\right)  }\right) \\
c_{t}  &  =\int_{0}^{t}b^{\left(  n\right)  }\left(  s,X_{s}^{\left(
1\right)  }\right)  \mathrm{d}s-\int_{0}^{t}b^{\left(  n\right)  }\left(
s,X_{s}^{\left(  2\right)  }\right)  \mathrm{d}s\\
d_{t}  &  =\int_{0}^{t}\left(  \nabla U^{\left(  n\right)  }\left(
s,X_{s}^{\left(  1\right)  }\right)  -\nabla U^{\left(  n\right)  }\left(
s,X_{s}^{\left(  2\right)  }\right)  \right)  \cdot\mathrm{d}W_{s}.
\end{align*}
We use the inequality%
\begin{equation}
\left\vert a_{t}+b_{t}+c_{t}+d_{t}\right\vert ^{p}\leq\frac{3}{2}\left\vert
b_{t}\right\vert ^{p}+C_{p}\left\vert a_{t}\right\vert ^{p}+C_{p}\left\vert
c_{t}\right\vert ^{p}+C_{p}\left\vert d_{t}\right\vert ^{p}.
\label{strange ineq}%
\end{equation}
Let us take $T\leq T_{0}$ given by the lemma. With new values of $C_{p}$ when
necessary, from the estimates of the lemma we have%
\[
E\left[  \sup_{t\in\left[  0,T\right]  }\left\vert a_{t}\right\vert
^{p}\right]  \leq C_{p}\left\vert x^{\left(  1\right)  }-x^{\left(  2\right)
}\right\vert ^{p}+C_{p}C_{n}^{p}\left(  T\right)  \left\vert x^{\left(
1\right)  }-x^{\left(  2\right)  }\right\vert ^{p}\leq C_{p}\left\vert
x^{\left(  1\right)  }-x^{\left(  2\right)  }\right\vert ^{p}%
\]
because%
\begin{align*}
\left\vert \sum_{k=0}^{n}U_{\mathcal{T}^{k}\left(  b\right)  }\left(
0,x^{\left(  1\right)  }\right)  -\sum_{k=0}^{n}U_{\mathcal{T}^{k}\left(
b\right)  }\left(  0,x^{\left(  2\right)  }\right)  \right\vert  &  \leq
\sum_{k=0}^{n}\left\vert U_{\mathcal{T}^{k}\left(  b\right)  }\left(
0,x^{\left(  1\right)  }\right)  -U_{\mathcal{T}^{k}\left(  b\right)  }\left(
0,x^{\left(  2\right)  }\right)  \right\vert \\
&  \leq\sum_{k=0}^{n}\left\Vert \nabla U_{\mathcal{T}^{k}\left(  b\right)
}\right\Vert _{L^{\infty}\left(  T\right)  }\left\vert x^{\left(  1\right)
}-x^{\left(  2\right)  }\right\vert \\
&  =C_{n}\left(  T\right)  \left\vert x^{\left(  1\right)  }-x^{\left(
2\right)  }\right\vert .
\end{align*}

\bigskip Similarly%
\[
E\left[  \sup_{t\in\left[  0,T\right]  }\left\vert b_{t}\right\vert
^{p}\right]  \leq\frac{1}{2^{p}}E\left[  \sup_{t\in\left[  0,T\right]
}\left\vert X_{t}^{\left(  1\right)  }-X_{t}^{\left(  2\right)  }\right\vert
^{p}\right]  .
\] 
These are the first two terms which contribute to estimate from above the
quantity (\ref{to be estimated}). The estimate of the third term $\left\vert
c_{t}\right\vert ^{p}$ is made by estimating the following two terms,
$i=1,2$,
\begin{align*}
E\left[  \sup_{t\in\left[  0,T\right]  }\left\vert \int_{0}^{t}\mathcal{T}%
^{n+1}\left(  b\right)  \left(  s,X_{s}^{\left(  i\right)  }\right)
\mathrm{d}s\right\vert ^{p}\right]  &\leq E\left[  \left(  \int_{0}%
^{T}\left\vert \mathcal{T}^{n+1}\left(  b\right)  \left(  s,X_{s}^{\left(
i\right)  }\right)  \right\vert \mathrm{d}s\right)  ^{p}\right]\\ 
&  \leq
C^{p}T^{p}2^{-\left(  n+1\right)  p}.
\end{align*}
Finally, the forth term is%
\[
E\left[  \sup_{t\in\left[  0,T\right]  }\left\vert d_{t}\right\vert
^{p}\right]  \leq C_{p}E\left[  \left(  \int_{0}^{T}\left\Vert \sum_{k=0}%
^{n}\left(  \nabla U_{\mathcal{T}^{k}\left(  b\right)  }\left(  s,X_{s}%
^{\left(  1\right)  }\right)  -\nabla U_{\mathcal{T}^{k}\left(  b\right)
}\left(  s,X_{s}^{\left(  2\right)  }\right)  \right)  \right\Vert
^{2}\mathrm{d}s\right)  ^{p/2}\right]  .
\]
We have%
\begin{align*}
&  \left\Vert \sum_{k=0}^{n}\left(  \nabla U_{\mathcal{T}^{k}\left(  b\right)
}\left(  s,X_{s}^{\left(  1\right)  }\right)  -\nabla U_{\mathcal{T}%
^{k}\left(  b\right)  }\left(  s,X_{s}^{\left(  2\right)  }\right)  \right)
\right\Vert \\
& \hspace{2cm}  \leq\sum_{k=0}^{n}\left\Vert \nabla U_{\mathcal{T}^{k}\left(  b\right)
}\left(  s,X_{s}^{\left(  1\right)  }\right)  -\nabla U_{\mathcal{T}%
^{k}\left(  b\right)  }\left(  s,X_{s}^{\left(  2\right)  }\right)
\right\Vert \\
& \hspace{2cm}  \leq\sum_{k=0}^{n}\left\Vert D^{2}U_{\mathcal{T}^{k}\left(  b\right)
}\right\Vert _{L^{\infty}\left(  T\right)  }\left\vert X_{s}^{\left(
1\right)  }-X_{s}^{\left(  2\right)  }\right\vert =D_{n}\left(  T\right)
\left\vert X_{s}^{\left(  1\right)  }-X_{s}^{\left(  2\right)  }\right\vert
\end{align*}
hence%
\begin{align*}
&  E\left[  \sup_{t\in\left[  0,T\right]  }\left\vert \int_{0}^{t}\left(
\sum_{k=0}^{n}\nabla U_{\mathcal{T}^{k}\left(  b\right)  }\left(
s,X_{s}^{\left(  1\right)  }\right)  -\sum_{k=0}^{n}\nabla U_{\mathcal{T}%
^{k}\left(  b\right)  }\left(  s,X_{s}^{\left(  2\right)  }\right)  \right)
\mathrm{d}W_{s}\right\vert ^{p}\right] \\
& \hspace{2cm} \leq C_{p}E\left[  \left(  \int_{0}^{T}\left\vert X_{s}^{\left(  1\right)
}-X_{s}^{\left(  2\right)  }\right\vert ^{2}\mathrm{d}s\right)  ^{p/2}\right]
\leq C_{p}T^{p/2}E\left[  \sup_{t\in\left[  0,T\right]  }\left\vert
X_{t}^{\left(  1\right)  }-X_{t}^{\left(  2\right)  }\right\vert ^{p}\right]
.
\end{align*}
Summarizing, using (\ref{strange ineq}) we have proved%
\begin{align*}
&  E\left[  \sup_{t\in\left[  0,T\right]  }\left\vert X_{t}^{\left(  1\right)
}-X_{t}^{\left(  2\right)  }\right\vert ^{p}\right] \\
&\hspace{1cm}  \leq C_{p}\left\vert x^{\left(  1\right)  }-x^{\left(  2\right)
}\right\vert ^{p}+\left(  \frac{3}{2}\frac{1}{2^{p}}+C_{p}T^{p/2}\right)
E\left[  \sup_{t\in\left[  0,T\right]  }\left\vert X_{t}^{\left(  1\right)
}-X_{t}^{\left(  2\right)  }\right\vert ^{p}\right]  +C^{p}T^{p}2^{-\left(
n+1\right)  p}.
\end{align*}
Since this is true for every $n$, we have%
\[
E\left[  \sup_{t\in\left[  0,T\right]  }\left\vert X_{t}^{\left(  1\right)
}-X_{t}^{\left(  2\right)  }\right\vert ^{p}\right]  \leq C_{p}\left\vert
x^{\left(  1\right)  }-x^{\left(  2\right)  }\right\vert ^{p}+\left(  \frac
{3}{2}\frac{1}{2^{p}}+C_{p}T^{p/2}\right)  E\left[  \sup_{t\in\left[
0,T\right]  }\left\vert X_{t}^{\left(  1\right)  }-X_{t}^{\left(  2\right)
}\right\vert ^{p}\right]  .
\]
Then there exists $T_{1}>0$ such that
\[
E\left[  \sup_{t\in\left[  0,T_{1}\right]  }\left\vert X_{t}^{\left(
1\right)  }-X_{t}^{\left(  2\right)  }\right\vert ^{p}\right]  \leq
C_{p}\left\vert x^{\left(  1\right)  }-x^{\left(  2\right)  }\right\vert
^{p}.
\]
This implies uniqueness and the existence of the modification (by Kolmogorov
regularity theorem for fields with values in $C\left(  \left[  0,T\right]
;\mathbb{R}\right)  $ and the arbitrariness of $p\geq2$), as claimed by the
theorem, over the time interval $\left[  0,T_{1}\right]  $. By classical
arguments one can iterate the result on successive intervals (their size does
not change), so the result is true over $\left[  0,T\right]  $. The proof in
the H\"{o}lder case is complete.

\begin{remark}
\label{remark lambda}One can work on the full initial interval $\left[
0,T\right]  $ from the beginning, by means of the following trick, developed
in \cite{FGP}: one takes the heat equation with damping%
\[
\frac{\partial u}{\partial t}+\frac{1}{2}\Delta u=\lambda u+\varphi\text{ on
}\left[  0,T\right]  ,\qquad u|_{t=T}=0
\]
and use the fact that for large $\lambda$ the gradient of $u$ is uniformly
small (like $\lambda^{-1/2}$).
\end{remark}

\section{General case}

Let us now go back to the general case when $b\in L_{p}^{q}\left(  T\right)
$. The main novelty is that $d_{t}$ cannot be estimated as above, since
$D^{2}U$ is not bounded. We use two tricks to overcome this apparently very
serious difficulty, used before in other works on uniqueness for certain
nonlinear equations: introduce a suitable increasing process $A_{t}$ related
to $D^{2}U$ (see \cite{Ver}) and pre-multiply by $e^{-A_{t}}$ (see
\cite{Schmal}, \cite{DD03}).

Let again $X_{t}^{\left(  i\right)  }$, $i=1,2$, be two solutions with initial
conditions $x^{\left(  i\right)  }$, $i=1,2$. Given any $p\geq2$, we want to
estimate (\ref{to be estimated}). We follow a different route with respect to
the previous section. Set, for $i=1,2$ and $n\in\mathbb{N}$,
\begin{align*}
Y_{t}^{\left(  i,n\right)  }  &  :=X_{t}^{\left(  i\right)  }+U^{\left(
n\right)  }\left(  t,X_{t}^{\left(  i\right)  }\right)  ,\\
b_{t}^{\left(  i,n\right)  }  &  :=b^{\left(  n\right)  }\left(
t,X_{t}^{\left(  i\right)  }\right)  ,\qquad\sigma_{t}^{\left(  i,n\right)
}:=\nabla U^{\left(  n\right)  }\left(  t,X_{t}^{\left(  i\right)  }\right)  .
\end{align*}
We drop the index $n$ in intermediate computations, when it is not essential
to emphasize the dependence on $n$. The equation reads now%
\[
Y_{t}^{\left(  i\right)  }=Y_{0}^{\left(  i\right)  }+\int_{0}^{t}%
b_{s}^{\left(  i\right)  }\mathrm{d}s+\int_{0}^{t}\left(  \sigma_{s}^{\left(
i\right)  }+I\right)  \cdot\mathrm{d}W_{s}.
\]
Controlling $\left\vert X_{t}^{\left(  1\right)  }-X_{t}^{\left(  2\right)
}\right\vert $ is the same as $\left\vert Y_{t}^{\left(  1,n\right)  }%
-Y_{t}^{\left(  2,n\right)  }\right\vert $, for each $n$, for small $T$ (this
reminds \cite[lemma 10.6]{KR05}, although the approach is different)

\begin{lemma}
\label{lemma 2}Recall the notation $C_{n}\left(  T\right)  $ from lemma
\ref{lemma 1}. There exists $T_{0}$ such that for all $T\in(0,$\bigskip
$T_{0}]$ we have
\begin{equation}
C_{n}\left(  T\right)  \leq\frac{1}{2},\qquad\left\Vert \mathcal{T}^{n}\left(
b\right)  \right\Vert _{L_{p}^{q}\left(  T\right)  }\leq C\frac{1}{2^{n}}.
\label{same bounds}%
\end{equation}
It follows also%
\begin{align*}
\left\vert Y_{t}^{\left(  1,n\right)  }-Y_{t}^{\left(  2,n\right)
}\right\vert  &  \leq\frac{3}{2}\left\vert X_{t}^{\left(  1\right)  }%
-X_{t}^{\left(  2\right)  }\right\vert \\
\left\vert X_{t}^{\left(  1\right)  }-X_{t}^{\left(  2\right)  }\right\vert
&  \leq2\left\vert Y_{t}^{\left(  1,n\right)  }-Y_{t}^{\left(  2,n\right)
}\right\vert
\end{align*}
for $t\in\left[  0,T\right]  $ and%
\[
\left\Vert \nabla U^{\left(  n\right)  }\right\Vert _{L^{\infty}\left(
T\right)  }\leq\frac{1}{2},\qquad\left\Vert D^{2}U^{\left(  n\right)
}\right\Vert _{L_{p}^{q}\left(  T\right)  }\leq C
\]
for some constant $C>0$.
\end{lemma}

\begin{proof}
Using (\ref{bounded gradient}) instead of (\ref{bounded gradient 2}) one can
prove (\ref{same bounds}) as in the case of lemma \ref{lemma 1}. The
modifications are that we use
\[
\varepsilon=\left(  4\left\Vert b\right\Vert _{L_{p}^{q}\left(  T\right)
}\right)  ^{-1}%
\]
and we get the inequalities%
\[
\left\Vert \nabla U_{b}\right\Vert _{L^{\infty}\left(  T\right)  }%
\leq\varepsilon\left\Vert b\right\Vert _{L_{p}^{q}\left(  T\right)  }%
\]%
\[
\left\Vert \mathcal{T}^{1}\left(  b\right)  \right\Vert _{L_{p}^{q}\left(
T\right)  }\leq\left\Vert \nabla U_{b}\right\Vert _{L^{\infty}\left(
T\right)  }\left\Vert b\right\Vert _{L_{p}^{q}\left(  T\right)  }%
\leq\varepsilon\left\Vert b\right\Vert _{L_{p}^{q}\left(  T\right)  }^{2}%
\]
and so on by iteration. We do not rewrite all the details. Having proved
(\ref{same bounds}), we have (using a simple approximation argument to write
the estimate with $\Vert \nabla U_{\mathcal{T}^{k}\left(  b\right)
}\Vert _{L^{\infty}\left(  T\right)  }$)
\begin{align*}
\left\vert Y_{t}^{\left(  1\right)  }-Y_{t}^{\left(  2\right)  }\right\vert
&  \leq\left\vert X_{t}^{\left(  1\right)  }-X_{t}^{\left(  2\right)
}\right\vert +\sum_{k=0}^{n}\left\Vert \nabla U_{\mathcal{T}^{k}\left(
b\right)  }\right\Vert _{L^{\infty}\left(  T\right)  }\left\vert
X_{t}^{\left(  1\right)  }-X_{t}^{\left(  2\right)  }\right\vert \\
&  =\left\vert X_{t}^{\left(  1\right)  }-X_{t}^{\left(  2\right)
}\right\vert +C_{n}\left(  T\right)  \left\vert X_{t}^{\left(  1\right)
}-X_{t}^{\left(  2\right)  }\right\vert \leq\frac{3}{2}\left\vert
X_{t}^{\left(  1\right)  }-X_{t}^{\left(  2\right)  }\right\vert
\end{align*}
and%
\begin{align*}
\left\vert X_{t}^{\left(  1\right)  }-X_{t}^{\left(  2\right)  }\right\vert
&  \leq\left\vert Y_{t}^{\left(  1\right)  }-Y_{t}^{\left(  2\right)
}\right\vert +\left\vert \sum_{k=0}^{n}U_{\mathcal{T}^{k}\left(  b\right)
}\left(  t,X_{t}^{\left(  1\right)  }\right)  -U_{\mathcal{T}^{k}\left(
b\right)  }\left(  t,X_{t}^{\left(  2\right)  }\right)  \right\vert \\
&  \leq\left\vert Y_{t}^{\left(  1\right)  }-Y_{t}^{\left(  2\right)
}\right\vert +\frac{1}{2}\left\vert X_{t}^{\left(  1\right)  }-X_{t}^{\left(
2\right)  }\right\vert
\end{align*}
and thus $\left\vert X_{t}^{\left(  1\right)  }-X_{t}^{\left(  2\right)
}\right\vert \leq2\left\vert Y_{t}^{\left(  1\right)  }-Y_{t}^{\left(
2\right)  }\right\vert $. Finally, $\left\Vert \nabla U^{\left(  n\right)
}\right\Vert _{L^{\infty}\left(  T\right)  }\leq C_{n}\left(  T\right)  $ and
\[
\left\Vert D^{2}U^{\left(  n\right)  }\right\Vert _{L_{p}^{q}\left(  T\right)
}\leq\sum_{k=0}^{n}\left\Vert D^{2}U_{\mathcal{T}^{k}\left(  b\right)
}\right\Vert _{L_{p}^{q}\left(  T\right)  }\leq C\sum_{k=0}^{n}\left\Vert
\mathcal{T}^{k}\left(  b\right)  \right\Vert _{L_{p}^{q}\left(  T\right)  }%
\]
by (\ref{maximal}), and the series converges by (\ref{same bounds}). The proof
is complete.
\end{proof}

By It\^{o} formula we have%
\begin{align*}
d\left\vert Y_{t}^{\left(  1\right)  }-Y_{t}^{\left(  2\right)  }\right\vert
^{p}  &  \leq p\left\vert Y_{t}^{\left(  1\right)  }-Y_{t}^{\left(  2\right)
}\right\vert ^{p-1}\left\vert b_{t}^{\left(  1\right)  }-b_{t}^{\left(
2\right)  }\right\vert \mathrm{d}t\\
&  +p\left\vert Y_{t}^{\left(  1\right)  }-Y_{t}^{\left(  2\right)
}\right\vert ^{p-2}\left\langle Y_{t}^{\left(  1\right)  }-Y_{t}^{\left(
2\right)  },\left(  \sigma_{t}^{\left(  1\right)  }-\sigma_{t}^{\left(
2\right)  }\right)  \cdot\mathrm{d}W_{t}\right\rangle \\
&  +C_{p}^{\ast}\left\vert Y_{t}^{\left(  1\right)  }-Y_{t}^{\left(  2\right)
}\right\vert ^{p-2}\left\Vert \sigma_{t}^{\left(  1\right)  }-\sigma
_{t}^{\left(  2\right)  }\right\Vert ^{2}\mathrm{d}t
\end{align*}
for a suitable constant $C_{p}^{\ast}$. Following Veretennikov \cite{Ver},
denote by $\frac{\left\Vert \sigma_{t}^{\left(  1\right)  }-\sigma
_{t}^{\left(  2\right)  }\right\Vert ^{2}}{\left\vert Y_{t}^{\left(  1\right)
}-Y_{t}^{\left(  2\right)  }\right\vert ^{2}}1_{\left\{  Y_{t}^{\left(
1\right)  }\neq Y_{t}^{\left(  2\right)  }\right\}  }$ the non negative
function equal to $\frac{\left\Vert \sigma_{t}^{\left(  1\right)  }-\sigma
_{t}^{\left(  2\right)  }\right\Vert ^{2}}{\left\vert Y_{t}^{\left(  1\right)
}-Y_{t}^{\left(  2\right)  }\right\vert ^{2}}$ when $Y_{t}^{\left(  1\right)
}\neq Y_{t}^{\left(  2\right)  }$ and equal to zero otherwise. Set
\[
A_{t}^{\left(  n\right)  }:=\int_{0}^{t}\frac{\left\Vert \sigma_{s}^{\left(
1,n\right)  }-\sigma_{s}^{\left(  2,n\right)  }\right\Vert ^{2}}{\left\vert
Y_{s}^{\left(  1,n\right)  }-Y_{s}^{\left(  2,n\right)  }\right\vert ^{2}%
}1_{\left\{  Y_{s}^{\left(  1,n\right)  }\neq Y_{s}^{\left(  2,n\right)
}\right\}  }\mathrm{d}s
\]
(we write $A_{t}$ when $n$ is not the main concern) which a priori may be
infinite. We shall prove below, lemma \ref{lemma 4}, that this is a finite,
even exponentially integrable uniformly in $n$, increasing non negative
process. Then%
\begin{align*}
d\left(  e^{-C_{p}^{\ast}A_{t}}\left\vert Y_{t}^{\left(  1\right)  }%
-Y_{t}^{\left(  2\right)  }\right\vert ^{p}\right)   &  \leq e^{-C_{p}^{\ast
}A_{t}}p\left\vert Y_{t}^{\left(  1\right)  }-Y_{t}^{\left(  2\right)
}\right\vert ^{p-1}\left\vert b_{t}^{\left(  1\right)  }-b_{t}^{\left(
2\right)  }\right\vert \mathrm{d}t\\
&  +e^{-C_{p}^{\ast}A_{t}}p\left\vert Y_{t}^{\left(  1\right)  }%
-Y_{t}^{\left(  2\right)  }\right\vert ^{p-2}\left\langle Y_{t}^{\left(
1\right)  }-Y_{t}^{\left(  2\right)  },\left(  \sigma_{t}^{\left(  1\right)
}-\sigma_{t}^{\left(  2\right)  }\right)  \cdot\mathrm{d}W_{t}\right\rangle \\
&  +e^{-C_{p}^{\ast}A_{t}}C_{p}^{\ast}\left\vert Y_{t}^{\left(  1\right)
}-Y_{t}^{\left(  2\right)  }\right\vert ^{p-2}\left\Vert \sigma_{t}^{\left(
1\right)  }-\sigma_{t}^{\left(  2\right)  }\right\Vert ^{2}\mathrm{d}t\\
&  -C_{p}^{\ast}e^{-C_{p}^{\ast}A_{t}}\left\vert Y_{t}^{\left(  1\right)
}-Y_{t}^{\left(  2\right)  }\right\vert ^{p}dA_{t}.
\end{align*}
Since
\[
e^{-C_{p}^{\ast}A_{t}}C_{p}^{\ast}\left\vert Y_{t}^{\left(  1\right)  }%
-Y_{t}^{\left(  2\right)  }\right\vert ^{p-2}\left\Vert \sigma_{t}^{\left(
1\right)  }-\sigma_{t}^{\left(  2\right)  }\right\Vert ^{2}\mathrm{d}%
t-C_{p}^{\ast}e^{-C_{p}^{\ast}A_{t}}\left\vert Y_{t}^{\left(  1\right)
}-Y_{t}^{\left(  2\right)  }\right\vert ^{p}dA_{t}=0
\]
The inequality simplifies to%
\begin{align*}
d\left(  e^{-C_{p}^{\ast}A_{t}}\left\vert Y_{t}^{\left(  1\right)  }%
-Y_{t}^{\left(  2\right)  }\right\vert ^{p}\right)   &  \leq e^{-C_{p}^{\ast
}A_{t}}p\left\vert Y_{t}^{\left(  1\right)  }-Y_{t}^{\left(  2\right)
}\right\vert ^{p-1}\left\vert b_{t}^{\left(  1\right)  }-b_{t}^{\left(
2\right)  }\right\vert \mathrm{d}t\\
&  +e^{-C_{p}^{\ast}A_{t}}p\left\vert Y_{t}^{\left(  1\right)  }%
-Y_{t}^{\left(  2\right)  }\right\vert ^{p-2}\left\langle Y_{t}^{\left(
1\right)  }-Y_{t}^{\left(  2\right)  },\left(  \sigma_{t}^{\left(  1\right)
}-\sigma_{t}^{\left(  2\right)  }\right)  \cdot\mathrm{d}W_{t}\right\rangle .
\end{align*}
The last term is a martingale:\ the processes $\sigma_{t}^{\left(  i\right)
}$ are bounded (recall that $\nabla U$ is bounded), $e^{-C_{p}^{\ast}A_{t}}$
is bounded by 1, and $\left\vert Y_{t}^{\left(  1\right)  }-Y_{t}^{\left(
2\right)  }\right\vert $ is integrable at any power, since it is smaller than
$\frac{3}{2}\left\vert X_{t}^{\left(  1\right)  }-X_{t}^{\left(  2\right)
}\right\vert $ (lemma \ref{lemma 2}) and we know that solutions of equation
(\ref{SDE}) are integrable to any power, see proposition
\ref{proposition powers} in the Appendix. Therefore, using also $\left\vert
Y_{0}^{\left(  1\right)  }-Y_{0}^{\left(  2\right)  }\right\vert \leq\frac
{3}{2}\left\vert x^{\left(  1\right)  }-x^{\left(  2\right)  }\right\vert $
(lemma \ref{lemma 2})%
\[
E\left[  e^{-C_{p}^{\ast}A_{t}}\left\vert Y_{t}^{\left(  1\right)  }%
-Y_{t}^{\left(  2\right)  }\right\vert ^{p}\right]  \leq C_{p}\left\vert
x^{\left(  1\right)  }-x^{\left(  2\right)  }\right\vert ^{p}+p\int_{0}%
^{t}E\left[  \left\vert Y_{s}^{\left(  1\right)  }-Y_{s}^{\left(  2\right)
}\right\vert ^{p-1}\left\vert b_{s}^{\left(  1,n\right)  }-b_{s}^{\left(
2,n\right)  }\right\vert \right]  \mathrm{d}s.
\]
Using again lemma \ref{lemma 2}, both in the first and last term, we get
\begin{equation}
E\hspace{-0,05cm}\left[  e^{-C_{p}^{\ast}A_{t}^{\left(  n\right)  }}\left\vert X_{t}^{\left(
1\right)  }\hspace{-0,05cm}-\hspace{-0,05cm}X_{t}^{\left(  2\right)  }\right\vert ^{p}\right]  \leq
C_{p}\left\vert x^{\left(  1\right)  }\hspace{-0,05cm}-x^{\left(  2\right)  }\right\vert
^{p}\hspace{-0,05cm}+C_{p}\hspace{-0,05cm}\int_{0}^{T}\hspace{-0,25cm}E\hspace{-0,05cm}\left[  \left\vert X_{s}^{\left(  1\right)  }\hspace{-0,05cm}%
-\hspace{-0,05cm}X_{s}^{\left(  2\right)  }\right\vert ^{p-1}\left\vert b_{s}^{\left(
1,n\right)  }-b_{s}^{\left(  2,n\right)  }\right\vert \right] \hspace{-0,05cm} \mathrm{d}s.
\label{intermediate}%
\end{equation}

\begin{lemma}
\label{lemma 3}For every $\alpha,\beta\geq1$,%
\[
\lim_{n\rightarrow\infty}E\left[  \left(  \int_{0}^{T}\left\vert
X_{s}^{\left(  1\right)  }-X_{s}^{\left(  2\right)  }\right\vert ^{\alpha
}\left\vert b_{s}^{\left(  1,n\right)  }-b_{s}^{\left(  2,n\right)
}\right\vert \mathrm{d}s\right)  ^{\beta}\right]  =0.
\]

\end{lemma}

\begin{proof}
We have%
\begin{align*}
&  E\left[  \left(  \int_{0}^{T}\left\vert X_{s}^{\left(  1\right)  }%
-X_{s}^{\left(  2\right)  }\right\vert ^{\alpha}\left\vert b_{s}^{\left(
1,n\right)  }-b_{s}^{\left(  2,n\right)  }\right\vert \mathrm{d}s\right)
^{\beta}\right] \\
&  \leq E\left[  \left(  \int_{0}^{T}\left\vert X_{s}^{\left(  1\right)
}-X_{s}^{\left(  2\right)  }\right\vert ^{2\alpha}\mathrm{d}s\right)  ^{\beta
}\right]  ^{1/2}E\left[  \left(  \int_{0}^{T}\left\vert b_{s}^{\left(
1,n\right)  }-b_{s}^{\left(  2,n\right)  }\right\vert ^{2}\mathrm{d}s\right)
^{\beta}\right]  ^{1/2}.
\end{align*}
The first term is bounded since $E\left[  \int_{0}^{T}\left\vert
X_{s}^{\left(  i\right)  }\right\vert ^{N}\mathrm{d}s\right]  <\infty$ for
each $N>0$, $i=1,2$, see proposition \ref{proposition powers} in the Appendix.
Let us prove that the second term converges to zero. For each $i=1,2$, we have%
\[
E\left[  \left(  \int_{0}^{T}\left\vert 2^{n+1}b_{s}^{\left(  i,n\right)
}\right\vert ^{2}\mathrm{d}s\right)  ^{\beta}\right]  =E\left[  \left(
\int_{0}^{T}\left\vert f_{n}\left(  s,X_{s}^{\left(  i\right)  }\right)
\right\vert ^{2}\mathrm{d}s\right)  ^{\beta}\right]
\]
where
\[
f_{n}:=2^{n+1}\mathcal{T}^{n+1}\left(  b\right)
\]
are equibounded in $L_{p}^{q}\left(  T\right)  $ by lemma \ref{lemma 2}. From
Girsanov formula (\ref{formula Gir 2}) of the Appendix we have%
\begin{align*}
&\hspace{-1,4cm}  E\left[  \left(  \int_{0}^{T}\left\vert f_{n}\left(  s,X_{s}^{\left(
i\right)  }\right)  \right\vert ^{2}\mathrm{d}s\right)  ^{\beta}\right] \\
&\hspace{-1,4cm}  =E\left[  \left(  \int_{0}^{T}\left\vert f_{n}\left(  s,x^{\left(
i\right)  }+W_{s}\right)  \right\vert ^{2}\mathrm{d}s\right)  ^{\beta
}e^{\,\int_{0}^{T}b(s,x^{\left(  i\right)  }+W_{s})\;\mathrm{d}W_{s}%
-1/2\int_{0}^{T}\left\vert b(s,x^{\left(  i\right)  }+W_{s})\right\vert
^{2}\mathrm{d}s}\right]  .
\end{align*}
This is equal to%
\begin{align*}
&  =E\left[  \left(  \int_{0}^{T}\left\vert f_{n}\left(  s,x+W_{s}\right)
\right\vert ^{2}\mathrm{d}s\right)  ^{\beta}e^{\int_{0}^{T}b(s,x+W_{s}%
)\,\mathrm{d}W_{s}-\frac{2}{2}\int_{0}^{T}\left\vert b(s,x+W_{s})\right\vert
^{2}\mathrm{d}s}e^{\frac{\left(  2-1\right)  }{2}\int_{0}^{T}\left\vert
b(s,x+W_{s})\right\vert ^{2}\mathrm{d}s}\right] \\
&  \leq E\left[  \left(  \int_{0}^{T}\left\vert f_{n}\left(  s,x+W_{s}\right)
\right\vert ^{2}\mathrm{d}s\right)  ^{2\beta}e^{\left(  2-1\right)  \int
_{0}^{T}\left\vert b(s,x+W_{s})\right\vert ^{2}\mathrm{d}s}\right]  ^{1/2}\\
&  \leq E\left[  \left(  \int_{0}^{T}\left\vert f_{n}\left(  s,x+W_{s}\right)
\right\vert ^{2}\mathrm{d}s\right)  ^{4\beta}\right]  ^{1/4}E\left[
e^{2\int_{0}^{T}\left\vert b(s,x+W_{s})\right\vert ^{2}\mathrm{d}s}\right]
^{1/4}%
\end{align*}
where we have used
\[
E\left[  e^{\int_{0}^{T}2b(s,x+W_{s})\,\mathrm{d}W_{s}-\frac{1}{2}\int_{0}%
^{T}\left\vert 2b(s,x+W_{s})\right\vert ^{2}\mathrm{d}s}\right]  =1.
\]
Both factors of the last inequality are bounded, by the exponential moment
estimates of corollary \ref{corollary appendix B}. Therefore we can find a
constant $K_{\beta}$ independent of $n$, such that%
\begin{equation}
E\left[  \left(  \int_{0}^{T}\left\vert f_{n}\left(  s,X_{s}^{\left(
i\right)  }\right)  \right\vert ^{2}\mathrm{d}s\right)  ^{\beta}\right]  \leq
K_{\beta} \label{ineq 1}%
\end{equation}
which implies $E\left[  \left(  \int_{0}^{T}\left\vert b_{s}^{\left(
i,n\right)  }\right\vert ^{2}\mathrm{d}s\right)  ^{\beta}\right]  \leq
K_{\beta}\frac{1}{\left(  2^{n+1}\right)  ^{2\beta}}$. The proof is complete.
\end{proof}

\bigskip From (\ref{intermediate}) and lemma \ref{lemma 3} we get%
\[
\underset{n\rightarrow\infty}{\lim\sup}\sup_{t\in\left[  0,T\right]  }E\left[
e^{-C_{p}^{\ast}A_{t}^{\left(  n\right)  }}\left\vert X_{t}^{\left(  1\right)
}-X_{t}^{\left(  2\right)  }\right\vert ^{p}\right]  \leq C_{p}\left\vert
x^{\left(  1\right)  }-x^{\left(  2\right)  }\right\vert ^{p}.
\]
But we have
\begin{align*}
E\left[  \left\vert X_{t}^{\left(  1\right)  }-X_{t}^{\left(  2\right)
}\right\vert ^{p/2}\right]   &  =E\left[  e^{C_{p}^{\ast}A_{t}^{\left(
n\right)  }/2}e^{-C_{p}^{\ast}A_{t}^{\left(  n\right)  }/2}\left\vert
X_{t}^{\left(  1\right)  }-X_{t}^{\left(  2\right)  }\right\vert ^{p/2}\right]
\\
&  \leq E\left[  e^{C_{p}^{\ast}A_{t}^{\left(  n\right)  }}\right]
^{1/2}E\left[  e^{-C_{p}^{\ast}A_{t}^{\left(  n\right)  }}\left\vert
X_{t}^{\left(  1\right)  }-X_{t}^{\left(  2\right)  }\right\vert ^{p}\right]
^{1/2}\\
&  \leq E\left[  e^{C_{p}^{\ast}A_{t}^{\left(  n\right)  }}\right]
^{1/2}C_{p}\left\vert x^{\left(  1\right)  }-x^{\left(  2\right)  }\right\vert
^{p/2}.
\end{align*}
From lemma \ref{lemma 4}, $E\left[  e^{C_{p}^{\ast}A_{t}^{\left(  n\right)  }%
}\right]  $ is uniformly bounded, so we include it into the constant and get
(renaming $p$)%
\[
\underset{n\rightarrow\infty}{\lim\sup}\sup_{t\in\left[  0,T\right]  }E\left[
\left\vert X_{t}^{\left(  1\right)  }-X_{t}^{\left(  2\right)  }\right\vert
^{p/2}\right]  \leq C_{p}\left\vert x^{\left(  1\right)  }-x^{\left(
2\right)  }\right\vert ^{p/2}.
\]
But now the left-hand-side is independent of $n$. We have proved the following
result, of independent interest. It is proved here for small $T$, but by
iteration or by the trick described in remark \ref{remark lambda}, it holds
true on the original time interval $\left[  0,T\right]  $.

\begin{proposition}
\label{proposition estimate}
\begin{equation}
\sup_{t\in\left[  0,T\right]  }E\left[  \left\vert X_{t}^{\left(  1\right)
}-X_{t}^{\left(  2\right)  }\right\vert ^{p}\right]  \leq C_{p}\left\vert
x^{\left(  1\right)  }-x^{\left(  2\right)  }\right\vert ^{p}.
\label{moment estimate for X}%
\end{equation}

\end{proposition}

Let us stress that, in our opinion, this proposition is a remarkable step
forward with respect to what was known before for equation (\ref{SDE}) under
$L_{p}^{q}$-drift. In a sense, the rest are more or less classical details.

Let us improve the proposition to an estimate for $E\left[  \sup_{t\in\left[
0,T\right]  }\left\vert X_{t}^{\left(  1\right)  }-X_{t}^{\left(  2\right)
}\right\vert ^{p}\right]  $. We may use the inequality proved above%
\begin{align*}
e^{-C_{p}^{\ast}A_{t}}\left\vert Y_{t}^{\left(  1\right)  }-Y_{t}^{\left(
2\right)  }\right\vert ^{p}  &  \leq\frac{3}{2}\left\vert x^{\left(  1\right)
}-x^{\left(  2\right)  }\right\vert ^{p} +p\int_{0}^{T}\left\vert Y_{s}^{\left(  1\right)  }-Y_{s}^{\left(
2\right)  }\right\vert ^{p-1}\left\vert b_{s}^{\left(  1\right)  }%
-b_{s}^{\left(  2\right)  }\right\vert \mathrm{d}t\\
&  +p\int_{0}^{t}\left\vert Y_{s}^{\left(  1\right)  }-Y_{s}^{\left(
2\right)  }\right\vert ^{p-2}\left\langle Y_{s}^{\left(  1\right)  }%
-Y_{s}^{\left(  2\right)  },\left(  \sigma_{s}^{\left(  1\right)  }-\sigma
_{s}^{\left(  2\right)  }\right)  \cdot\mathrm{d}W_{t}\right\rangle
\end{align*}
square it
\begin{align*}
e^{-2C_{p}^{\ast}A_{t}}\left\vert Y_{t}^{\left(  1\right)  }-Y_{t}^{\left(
2\right)  }\right\vert ^{2p}  &  \leq C\left\vert x^{\left(  1\right)
}-x^{\left(  2\right)  }\right\vert ^{2p}  +C_{p}\left(  \int_{0}^{T}\left\vert Y_{s}^{\left(  1\right)  }%
-Y_{s}^{\left(  2\right)  }\right\vert ^{p-1}\left\vert b_{s}^{\left(
1\right)  }-b_{s}^{\left(  2\right)  }\right\vert \mathrm{d}t\right)  ^{2}\\
&  +C_{p}\left(  \int_{0}^{t}\left\vert Y_{s}^{\left(  1\right)  }%
-Y_{s}^{\left(  2\right)  }\right\vert ^{p-2}\left\langle Y_{s}^{\left(
1\right)  }-Y_{s}^{\left(  2\right)  },\left(  \sigma_{s}^{\left(  1\right)
}-\sigma_{s}^{\left(  2\right)  }\right)  \cdot\mathrm{d}W_{t}\right\rangle
\right)  ^{2}%
\end{align*}
and apply Doob's inequality, and lemma \ref{lemma 2}, to get%
\begin{align*}
&  E\left[  \sup_{t\in\left[  0,T\right]  }\left(  e^{-2C_{p}^{\ast}%
A_{t}^{\left(  n\right)  }}\left\vert X_{t}^{\left(  1\right)  }%
-X_{t}^{\left(  2\right)  }\right\vert ^{2p}\right)  \right] \\
& \hspace{2cm} \leq2^{2p}E\left[  \sup_{t\in\left[  0,T\right]  }\left(  e^{-2C_{p}^{\ast
}A_{t}^{\left(  n\right)  }}\left\vert Y_{t}^{\left(  1\right)  }%
-Y_{t}^{\left(  2\right)  }\right\vert ^{2p}\right)  \right] \\
& \hspace{2cm}  \leq C_{p}\left\vert x^{\left(  1\right)  }-x^{\left(  2\right)
}\right\vert ^{2p}  +C_{p}E\left[  \left(  \int_{0}^{T}\left\vert Y_{s}^{\left(  1\right)
}-Y_{s}^{\left(  2\right)  }\right\vert ^{p-1}\left\vert b_{s}^{\left(
1\right)  }-b_{s}^{\left(  2\right)  }\right\vert \mathrm{d}t\right)
^{2}\right] \\
& \hspace{3cm} +C_{p}E\left[  \int_{0}^{T}\left\vert Y_{s}^{\left(  1\right)  }%
-Y_{s}^{\left(  2\right)  }\right\vert ^{2\left(  p-1\right)  }\left\Vert
\sigma_{s}^{\left(  1\right)  }-\sigma_{s}^{\left(  2\right)  }\right\Vert
^{2}\mathrm{d}s\right]  .
\end{align*}
One one side we have%
\begin{align*}
&  E\left[  \left(  \int_{0}^{T}\left\vert Y_{s}^{\left(  1,n\right)  }%
-Y_{s}^{\left(  2,n\right)  }\right\vert ^{p-1}\left\vert b_{s}^{\left(
1,n\right)  }-b_{s}^{\left(  2,n\right)  }\right\vert \mathrm{d}t\right)
^{2}\right] \\
& \hspace{3cm}  \leq C_{p}E\left[  \left(  \int_{0}^{T}\left\vert X_{s}^{\left(  1\right)
}-X_{s}^{\left(  2\right)  }\right\vert ^{p-1}\left\vert b_{s}^{\left(
1,n\right)  }-b_{s}^{\left(  2,n\right)  }\right\vert \mathrm{d}t\right)
^{2}\right]
\end{align*}
which converges to zero as $n\rightarrow\infty$, by lemma \ref{lemma 3}. On
the other side, since by definition of $\sigma_{t}^{\left(  i,n\right)  }$ and
inequality (\ref{same bounds}) we have $\left\vert \sigma_{s}^{\left(
i\right)  }\right\vert \leq\frac{1}{2}$, one has the estimate%
\begin{align*}
  E\left[  \int_{0}^{T}\left\vert Y_{s}^{\left(  1\right)  }-Y_{s}^{\left(
2\right)  }\right\vert ^{2\left(  p-1\right)  }\left\Vert \sigma_{s}^{\left(
1\right)  }-\sigma_{s}^{\left(  2\right)  }\right\Vert ^{2}\mathrm{d}s\right]
& \leq CE\left[  \int_{0}^{T}\left\vert Y_{s}^{\left(  1\right)  }%
-Y_{s}^{\left(  2\right)  }\right\vert ^{2\left(  p-1\right)  }\mathrm{d}%
s\right] \\
&  \leq C^{\prime}E\left[  \int_{0}^{T}\left\vert X_{s}^{\left(
1\right)  }-X_{s}^{\left(  2\right)  }\right\vert ^{2\left(  p-1\right)
}\mathrm{d}s\right] \\
&  \leq C_{p}\left\vert x^{\left(  1\right)  }-x^{\left(
2\right)  }\right\vert ^{2\left(  p-1\right)  }%
\end{align*}
by means of (\ref{moment estimate for X}). Summarizing and taking the limit as
$n\rightarrow\infty$ we have%
\[
\underset{n\rightarrow\infty}{\lim\sup}E\left[  \sup_{t\in\left[  0,T\right]
}\left(  e^{-2C_{p}^{\ast}A_{t}^{\left(  n\right)  }}\left\vert X_{t}^{\left(
1\right)  }-X_{t}^{\left(  2\right)  }\right\vert ^{2p}\right)  \right]  \leq
C\left\vert x^{\left(  1\right)  }-x^{\left(  2\right)  }\right\vert
^{2p}+C_{p}\left\vert x^{\left(  1\right)  }-x^{\left(  2\right)  }\right\vert
^{2\left(  p-1\right)  }.
\]
Moreover,%
\[
E\left[  e^{-2C_{p}^{\ast}A_{T}^{\left(  n\right)  }}\sup_{t\in\left[
0,T\right]  }\left(  \left\vert X_{t}^{\left(  1\right)  }-X_{t}^{\left(
2\right)  }\right\vert ^{2p}\right)  \right]  \leq E\left[  \sup_{t\in\left[
0,T\right]  }\left(  e^{-2C_{p}^{\ast}A_{t}^{\left(  n\right)  }}\left\vert
X_{t}^{\left(  1\right)  }-X_{t}^{\left(  2\right)  }\right\vert ^{2p}\right)
\right]
\]
and finally%
\begin{align*}
E\left[  \sup_{t\in\left[  0,T\right]  }\left(  \left\vert X_{t}^{\left(
1\right)  }-X_{t}^{\left(  2\right)  }\right\vert ^{p}\right)  \right]   &
=E\left[  e^{C_{p}^{\ast}A_{T}^{\left(  n\right)  }}e^{-C_{p}^{\ast}%
A_{T}^{\left(  n\right)  }}\sup_{t\in\left[  0,T\right]  }\left(  \left\vert
X_{t}^{\left(  1\right)  }-X_{t}^{\left(  2\right)  }\right\vert ^{p}\right)
\right] \\
&  \leq E\left[  e^{2C_{p}^{\ast}A_{T}^{\left(  n\right)  }}\right]
^{1/2}E\left[  e^{-2C_{p}^{\ast}A_{T}^{\left(  n\right)  }}\sup_{t\in\left[
0,T\right]  }\left(  \left\vert X_{t}^{\left(  1\right)  }-X_{t}^{\left(
2\right)  }\right\vert ^{2p}\right)  \right]  ^{1/2}.
\end{align*}
By lemma \ref{lemma 4} below, the previous inequalities give us%
\[
E\left[  \sup_{t\in\left[  0,T\right]  }\left(  \left\vert X_{t}^{\left(
1\right)  }-X_{t}^{\left(  2\right)  }\right\vert ^{p}\right)  \right]  \leq
C_{p}\left(  \left\vert x^{\left(  1\right)  }-x^{\left(  2\right)
}\right\vert ^{2p}+\left\vert x^{\left(  1\right)  }-x^{\left(  2\right)
}\right\vert ^{2\left(  p-1\right)  }\right)  ^{1/2}.
\]
By Kolmogorov theorem, we deduce the pathwise properties of our main theorem.
To complete the proof we need the following exponential estimate. The $L^{1}%
$-integrability of an expression very similar to $A_{T}^{\left(  n\right)  }$
has been proved in \cite{KR05}.

\begin{lemma}
\label{lemma 4} For any $k\in\mathbb{R}$ there is a constant $C_{k}>0$ such
that
\begin{equation}
E\left[  e^{kA_{T}^{\left(  n\right)  }}\right]  \leq C_{k}\,
\end{equation}
uniformly in $n\in\mathbb{N}$.
\end{lemma}

\vspace{0,2cm}

\textit{Proof: }For $Y_{s}^{(1)}\neq Y_{s}^{(2)}$, we also have $X_{s}%
^{(1)}\neq X_{s}^{(2)}$, by lemma \ref{lemma 2} (and vice versa, so the
functions $1_{\left\{  Y_{s}^{(1)}\neq Y_{s}^{(2)}\right\}  }$ and
$1_{\left\{  X_{s}^{(1)}\neq X_{s}^{(2)}\right\}  }$ coincide), so we may also
write%
\begin{align*}
\frac{\left\Vert \sigma_{s}^{\left(  1\right)  }-\sigma_{s}^{\left(  2\right)
}\right\Vert ^{2}}{\left\vert Y_{s}^{\left(  1\right)  }-Y_{s}^{\left(
2\right)  }\right\vert ^{2}}  &  =\frac{\left\Vert \nabla U^{\left(  n\right)
}\left(  s,X_{s}^{\left(  1\right)  }\right)  -\nabla U^{\left(  n\right)
}\left(  s,X_{s}^{\left(  2\right)  }\right)  \right\Vert ^{2}}{\left\vert
X_{s}^{(1)}-X_{s}^{(2)}\right\vert ^{2}}\frac{\left\vert X_{s}^{(1)}%
-X_{s}^{(2)}\right\vert ^{2}}{\left\vert Y_{s}^{\left(  1\right)  }%
-Y_{s}^{\left(  2\right)  }\right\vert ^{2}}\\
&  \leq2\frac{\left\Vert \nabla U^{\left(  n\right)  }\left(  s,X_{s}^{\left(
1\right)  }\right)  -\nabla U^{\left(  n\right)  }\left(  s,X_{s}^{\left(
2\right)  }\right)  \right\Vert ^{2}}{\left\vert X_{s}^{(1)}-X_{s}%
^{(2)}\right\vert ^{2}}%
\end{align*}
where we have used again lemma \ref{lemma 2}. Thus it is sufficient to prove
that%
\[
E\left[  \exp\left(  k\int_{0}^{T}\frac{\left\Vert \nabla U^{\left(  n\right)
}\left(  s,X_{s}^{\left(  1\right)  }\right)  -\nabla U^{\left(  n\right)
}\left(  s,X_{s}^{\left(  2\right)  }\right)  \right\Vert ^{2}}{\left\vert
X_{s}^{(1)}-X_{s}^{(2)}\right\vert ^{2}}\,1_{\left\{  X_{s}^{(1)}\neq
X_{s}^{(2)}\right\}  }\mathrm{d}s\right)  \right]  \leq C_{k}%
\]
where $C_{k}$ is a constant independent of $n$. Notice that $\nabla U^{\left(
n\right)  }$ are equibounded in \linebreak $L^{q}\left(0,T;W^{1,p}(\mathbb{R}
^{d})\right)$ by the last assertion of lemma \ref{lemma 2}. Thus, by the
density of $C_{c}^{\infty}\left(  [0,T]\times\mathbb{R}^{d}\right)  $ in
$L^{q}\left(0,T;W^{1,p}(  \mathbb{R}^{d}) \right )$, it is sufficient to prove
the following claim: for all smooth functions $f\in C_{c}^{\infty}\left(
[0,T]\times\mathbb{R}^{d}\right)  $ with $\Vert f\Vert_{L^{q}\left(0,T;W^{1,p}%
(  \mathbb{R}^{d}) \right )}\leq R$ we have
\begin{equation}
E\left[  \exp\left(  k\int_{0}^{T}\frac{\left\vert f\left(  s,X_{s}^{\left(
1\right)  }\right)  -f\left(  s,X_{s}^{\left(  2\right)  }\right)  \right\vert
^{2}}{\left\vert X_{s}^{(1)}-X_{s}^{(2)}\right\vert ^{2}}\,1_{\left\{
X_{s}^{(1)}\neq X_{s}^{(2)}\right\}  }\mathrm{d}s\right)  \right]  \leq
C_{k,R} \label{10 claim lemma e^A_t}%
\end{equation}
where $C_{k,R}$ depends only on $k$ and $R$.

For smooth functions $f$ we have%
\[
\frac{\left\vert f\left(  s,X_{s}^{\left(  1\right)  }\right)  -f\left(
s,X_{s}^{\left(  2\right)  }\right)  \right\vert ^{2}}{\left\vert X_{s}%
^{(1)}-X_{s}^{(2)}\right\vert ^{2}}\,1_{\left\{  X_{s}^{(1)}\neq X_{s}%
^{(2)}\right\}  }\leq\int_{0}^{1}\left\Vert \nabla f\left(  s,rX_{s}%
^{(1)}+(1-r)X_{s}^{(2)}\right)  \right\Vert ^{2}\mathrm{d}r.
\]
Using the convexity of the exponential function, we obtain that the left--hand
side of (\ref{10 claim lemma e^A_t}) is less than a constant times
\begin{equation}
\int_{0}^{1}E\left[  \exp\left(  k\int_{0}^{T}\left\Vert \nabla f\left(
s,rX_{s}^{(1)}+(1-r)X_{s}^{(2)}\right)  \right\Vert ^{2}\,\mathrm{d}s\right)
\right]  \,\mathrm{d}r. \label{int dr}%
\end{equation}
With the notations%
\begin{align*}
X_{s}^{\left(  r\right)  }  &  =rX_{s}^{(1)}+(1-r)X_{s}^{(2)},\qquad
x^{\left(  r\right)  }=rx^{\left(  1\right)  }+(1-r)x^{\left(  2\right)  }\\
b_{s}^{\left(  r\right)  }  &  =rb\left(  s,X_{s}^{(1)}\right)  +(1-r)b\left(
s,X_{s}^{(2)}\right)
\end{align*}
the process $X_{t}^{\left(  r\right)  }$ is given by
\[
X_{t}^{\left(  r\right)  }=x^{\left(  r\right)  }+\int_{0}^{t}b_{s}^{\left(
r\right)  }\,\mathrm{d}s+W_{t}.
\]
We have%
\[
E\left[  e^{\lambda\int_{0}^{T}\left\vert b_{t}^{\left(  r\right)
}\right\vert ^{2}\,\mathrm{d}t}\right]  \leq E\left[  e^{2\lambda r^{2}%
\int_{0}^{T}\left\vert b\left(  t,X_{t}^{(1)}\right)  \right\vert
^{2}\,\mathrm{d}t}e^{2\lambda(1-r)^{2}\int_{0}^{T}\left\vert b\left(
t,X_{t}^{(2)}\right)  \right\vert ^{2}\,\mathrm{d}t}\right]
\]
which is finite (by H\"{o}lder inequality) using the exponential estimates on
solutions of equation (\ref{SDE}) proved in the Appendix, see
(\ref{Novikov b(X)}). Hence Novikov condition is fulfilled;\ by Girsanov
theorem, $X_{t}^{\left(  r\right)  }$ is a Brownian motion from $x^{\left(
r\right)  }$, on $\left(  \Omega,F_{t},Q^{\left(  r\right)  }\right)  $ with
\[
\left.  \frac{dQ^{\left(  r\right)  }}{dP}\right\vert _{F_{T}}=\rho
_{T}^{\left(  r\right)  }:=\exp\left(  -\int_{0}^{T}b_{t}^{\left(  r\right)
}\ \cdot\mathrm{d}W_{t}-\frac{1}{2}\int_{0}^{T}\left\vert b_{t}^{\left(
r\right)  }\right\vert ^{2}\,\mathrm{d}t\right)  .
\]
Therefore we obtain (we indicate by superscripts the measure used in the
expected values)%
\begin{align*}
E^{P}\hspace{-0,05cm}\left[  \exp\hspace{-0,1cm}\left( \hspace{-0,05cm} k\hspace{-0,1cm}\int_{0}^{T}\hspace{-0,1cm} \left\Vert \nabla f\left(
s,X_{s}^{\left(  r\right)  }\right)  \right\Vert ^{2}\hspace{-0,1cm}\mathrm{d}s \hspace{-0,05cm}\right)
\right] \hspace{-0,1cm} &=\hspace{-0,1cm} E^{P}\hspace{-0,05cm}\left[  \left(  \rho_{T}^{\left(  r\right)  }\right)
^{-1/2}\hspace{-0,1cm}\left(  \rho_{T}^{\left(  r\right)  }\right)  ^{1/2}\hspace{-0,3cm}\exp\hspace{-0,1cm}\left( \hspace{-0,05cm}
k\hspace{-0,1cm}\int_{0}^{T}\hspace{-0,1cm} \left\Vert \nabla f\left(  s,X_{s}^{\left(  r\right)  }\right)
\right\Vert ^{2}\hspace{-0,1cm}\mathrm{d}s \hspace{-0,05cm}\right)  \right]
\\
&\leq CE^{P}\left[  \rho_{T}^{\left(  r\right)  }\exp\left(  2k\int_{0}^{T}\left\Vert \nabla f\left(  s,X_{s}^{\left(  r\right)  }\right) \right\Vert ^{2}\,\mathrm{d}s\right)  \right]  ^{1/2}\\
&=E^{Q}\left[  \exp\left( 2k\int_{0}^{T}\left\Vert \nabla f\left(  s,x^{\left(  r\right)  }
+W_{s}\right)  \right\Vert ^{2}\,\mathrm{d}s\right)  \right]  ^{1/2}.
\end{align*}
This is bounded by a constant depending only on the $L_{p}^{q}$ norm of
$\nabla f$, and on $k$, see corollary \ref{corollary appendix B} in the
Appendix. The proof is complete.

\section{Appendix}

We collect here known results, taken from the paper \cite{KR05} and previous
works, see for instance \cite{Port}, \cite{Ver}. They include weak existence
of a solution $X$ by Girsanov theorem, a formula for the density of the law of
the solution with respect to Wiener measure, weak uniqueness and the
exponential integrability of the process $\left\vert f(t,X_{t})\right\vert
^{2}$ when $f\in L_{p}^{q}\left(  T\right)  $ with $\frac{d}{p}+\frac{2}{q}<1$.

\begin{lemma}
\label{lemma f(W)} Given $p^{\prime},\,q^{\prime}\in\lbrack1,\infty]$ such
that
\begin{equation}
\frac{d}{p^{\prime}}+\frac{2}{q^{\prime}}<2 \label{pq<2}%
\end{equation}
there exist two positive constants $C$ and $\beta$ (it is $2\beta
=2-2/q^{\prime}-d/p^{\prime}$) with the following property: for every $f\in
L_{p^{\prime}}^{q^{\prime}}\left(  T\right)  $ and every $t>s$, $t,s\in\left[
0,T\right]  $,
\begin{equation}
\sup_{x\in\mathbb{R}^{d}}E\left[  \int_{s}^{t}f\left(  r,x+W_{r-s}\right)
\mathrm{d}r\right]  \leq C(t-s)^{\beta}\,\Vert f\Vert_{L_{p^{\prime}%
}^{q^{\prime}}\left(  T\right)  }. \label{stima E f 1}%
\end{equation}

\end{lemma}

\vspace{0cm}The proof is elementary (we write it only for $p^{\prime
},\,q^{\prime}\in\left(  1,\infty\right)  $): with $\frac{1}{p}+\frac
{1}{p^{\prime}}=1$, $\frac{1}{q}+\frac{1}{q^{\prime}}=1$, since%
\[
\int_{\mathbb{R}^{d}}\left(  2\pi(r-s)\right)  ^{-pd/2}\,e^{\frac{-p\left\vert
y\right\vert ^{2}}{2(r-s)}}\,\mathrm{d}y=C\left(  r-s\right)  ^{\left(
1-p\right)  d/2}%
\]
(we denote by $C$ a generic constant) and $\frac{q\left(  1-p\right)
d+2p}{2pq}=-\frac{d}{2p^{\prime}}+1-\frac{1}{q^{\prime}}$ we have
\begin{align*}
E\left[  \int_{s}^{t} \hspace{-0,1cm} f\left(  r,x+W_{r-s}\right)  \mathrm{d}r\right]   &
\leq\int_{s}^{t}\left(  \int_{\mathbb{R}^{d}} \hspace{-0,1cm}  f^{p^{\prime}}(r,y)\,\mathrm{d} y\right)  ^{1/p^{\prime}}\hspace{-0,1cm} \left(  \int_{\mathbb{R}^{d}} \hspace{-0,1cm} \left(  2\pi (r-s)\right)  ^{-pd/2}\,e^{\frac{-p\left\vert y\right\vert ^{2}}{2(r-s)} }\,\mathrm{d}y\right)  ^{1/p}\hspace{-0,2cm}\mathrm{d}r\\
&  \leq C\Vert f\Vert_{L_{p^{\prime}}^{q^{\prime}}\left(  T\right)  }\left(
\int_{s}^{t}\left(  r-s\right)  ^{q\left(  1-p\right)  d/2p}\mathrm{d} r\right)  ^{1/q}\\
&=C\Vert f\Vert_{L_{p^{\prime}}^{q^{\prime}}\left(  T\right)
}\left(  t-s\right)  ^{1-1/q^{\prime}-d/2p^{\prime}}.
\end{align*}

\begin{remark}
\label{oss pq-p'q'} As a consequence, if $f\in L_{p}^{q}\left(  T\right)  $
with $\frac{d}{p}+\frac{2}{q}<1$ (condition (\ref{pq})), then $f^{2}\in 
L_{p^{\prime}}^{q^{\prime}}\left(  T\right)  $  with $q^{\prime}=q/2$,
$p^{\prime}=p/2$ satisfying $\frac{d}{p^{\prime}}+\frac{2}{q^{\prime}}<2$, and
$\Vert f^{2}\Vert_{L_{p^{\prime}}^{q^{\prime}}\left(  T\right)  }\leq\Vert
f\Vert_{L_{p}^{q}\left(  T\right)  }^{2}$. Therefore
\[
\sup_{x\in\mathbb{R}^{d}}E\left[  \int_{s}^{t}f^{2}\left(  r,x+W_{r-s}\right)
\mathrm{d}r\right]  \leq C(t-s)^{\beta}\,\Vert f\Vert_{L_{p}^{q}\left(
T\right)  }^{2}.
\]

\end{remark}

\begin{lemma}
[Khas'minskii]\label{lemma Khas} Let $f:\mathbb{R}^{d}\rightarrow\mathbb{R}$
be a positive Borel function such that
\begin{equation}
\alpha:=\sup_{x\in\mathbb{R}^{d}}E\left[  \int_{0}^{T}f(s,x+W_{s}%
)\,\mathrm{d}s\right]  <1. \label{ipotesi Khas}%
\end{equation}
Then
\[
\sup_{x\in\mathbb{R}^{d}}E\left[  e^{\int_{0}^{T}f(s,x+W_{s})\,\mathrm{d}%
s}\right]  \leq\frac{1}{1-\alpha}.
\]

\end{lemma}

\vspace{0cm}See \cite{Kha59} or \cite[Chapter 1, lemma 2.1]{Sz98}.

\begin{corollary}
\label{corollary appendix B}If $f$ is a vector field of class $L_{p}%
^{q}\left(  T\right)  $ for some $p,\,q\in\lbrack1,\infty]$ such that
(\ref{pq}) holds, then there exists a constant $K_{f}$ depending on $\Vert
f\Vert_{L_{p}^{q}\left(  T\right)  }$ such that
\[
\sup_{x\in\mathbb{R}^{d}}E\left[  e^{\int_{0}^{T}\left\vert f(s,x+W_{s}%
)\,\right\vert ^{2}\mathrm{d}s}\right]  \leq K_{f}\,.
\]
Moreover, all (positive and negative) moments of
\begin{equation}
\rho_{T}:=\exp\left(  \int_{0}^{T}f(s,x+W_{s})\cdot\mathrm{d}W_{s}-\frac{1}%
{2}\int_{0}^{T}\left\vert f(s,x+W_{s})\right\vert ^{2}\mathrm{d}s\right)
\label{rho di lemma rho}%
\end{equation}
are finite.
\end{corollary}

\begin{proof}
Since $f\in L_{p}^{q}\left(  T\right)  $ with $p,\,q$ satisfying (\ref{pq}),
$f^{2}\in L_{p^{\prime}}^{q^{\prime}}\left(  T\right)  $ with $p^{\prime}%
=p/2$, $q^{\prime}=q/2$ satisfying (\ref{pq<2}). Since (\ref{pq<2}) is a
strict inequality, we may choose $\delta>0$ such that $\left\vert f\right\vert
^{2+\delta}\in L_{p^{\prime}}^{q^{\prime}}\left(  T\right)  $ for some new
$p^{\prime},\,q^{\prime}$ satisfying (\ref{pq<2}). Then we have inequality
(\ref{stima E f 1}) with $f$ replaced by $\left\vert f\right\vert ^{2+\delta}%
$. Choose $\varepsilon>0$ such that%
\[
\sup_{x\in\mathbb{R}^{d}}E\left[  \int_{0}^{T}\varepsilon\left\vert
f\right\vert ^{2+\delta}(s,x+W_{s})\,\mathrm{d}s\right]  <1.
\]
Then, by Khas'minskii lemma,
\[
\sup_{x\in\mathbb{R}^{d}}E\left[  e^{\int_{0}^{T}\varepsilon\left\vert
f\right\vert ^{2+\delta}(s,x+W_{s})\,\mathrm{d}s}\right]  <\infty.
\]
From Young inequality, there exists a constant $C_{\varepsilon,\delta}>0$ such
that $f^{2}\leq\varepsilon\left\vert f\right\vert ^{2+\delta}+C_{\varepsilon
,\delta}$. Then%
\[
\sup_{x\in\mathbb{R}^{d}}E\left[  e^{\int_{0}^{T}f^{2}(s,x+W_{s})\mathrm{d}%
s}\right]  \leq\sup_{x\in\mathbb{R}^{d}}E\left[  e^{\int_{0}^{T}%
\varepsilon\left\vert f\right\vert ^{2+\delta}(s,x+W_{s})\mathrm{d}s}\right]
e^{C_{\varepsilon,\delta}}<\infty.
\]
By inspection into the previous inequalities, we see that this bound depends
only on $\Vert f\Vert_{L_{p}^{q}\left(  T\right)  }$.

For the last claim, notice that, by Novikov condition, the process \linebreak $\rho
_{t}=\exp\left(  \int_{0}^{t}f(s,x+W_{s})\,\cdot\mathrm{d}W_{s}-\frac{1}%
{2}\int_{0}^{t}\left\vert f(s,x+W_{s})\right\vert ^{2}\mathrm{d}s\right)  $ is
an exponential martingale, in particular with $E\left[  \rho_{T}\right]  =1$.
Take any $\alpha>0$ and set $\overline{f}=2\alpha f$. This is again an element
of $L_{p}^{q}\left(  T\right)  $. Then we can define the corresponding
exponential martingale $\bar{\rho}$ with $\bar{b}$ in place of $b$, with
$E\left[  \bar{\rho}_{T}\right]  =1$. Then, for $\beta$ such that
$\sqrt{2\alpha\beta}=2\alpha$,
\begin{align*}
E\left[  \rho_{T}^{\alpha}\right]   &  =E\left[  e^{\int_{0}^{T}\alpha
f(s,x+W_{s})\cdot\mathrm{d}W_{s}-\frac{\alpha\beta}{2}\int_{0}^{T}\left\vert
f(s,x+W_{s})\right\vert ^{2}\mathrm{d}s}e^{\frac{\alpha\left(  \beta-1\right)
}{2}\int_{0}^{T}\left\vert f(s,x+W_{s})\right\vert ^{2}\mathrm{d}s}\right] \\
&  \leq E\left[  e^{\int_{0}^{T}2\alpha f(s,x+W_{s})\cdot\mathrm{d}W_{s}%
-\frac{1}{2}\int_{0}^{T}\left\vert \sqrt{2\alpha\beta}f(s,x+W_{s})\right\vert
^{2}\mathrm{d}s}\right]  ^{1/2}E\left[  e^{\alpha\left(  \beta-1\right)
\int_{0}^{T}\left\vert f(s,x+W_{s})\right\vert ^{2}\mathrm{d}s}\right]  ^{1/2}%
\end{align*}
which is finite since the first factor is $E\left[  \bar{\rho}_{T}\right]
^{1/2}=1$ and the second is finite by the first claim of the corollary applied
to $\sqrt{\left\vert \alpha\left(  \beta-1\right)  \right\vert }f\in L_{p}%
^{q}\left(  T\right)  $. \vspace{0cm}For $\alpha<0$ the computations are
similar. The proof is complete.
\end{proof}

By a classical application of Girsanov theorem (see \cite[lemma 3.2]{KR05} for
details) we have:

\begin{proposition}
\label{teo weak existence} Given $b\in L_{p}^{q}\left(  T\right)  $ with
$p,\,q$ satisfying (\ref{pq}) and $x\in\mathbb{R}^{d}$, there exist processes
$X_{t}$, $W_{t}$ defined for $t\in\lbrack0,T]$ on a filtered space
$(\Omega,F,F_{t},P)$ such that $W_{t}$ is a $d$--dimensional $\{F_{t}%
\}$--Wiener process and $X_{t}$ is an $\{F_{t}\}$--adapted,
continuous, $d$--dimensional process for which
\begin{equation}
P\left(  \int_{0}^{T}\left\vert b(t,X_{t})\right\vert ^{2}\ \mathrm{d}%
t<\infty\right)  =1 \label{condizione integ b}%
\end{equation}
and almost surely, for all $t\in\lbrack0,T]$%
\[
X_{t}=x+\int_{0}^{t}b(s,X_{s})\,\mathrm{d}s+W_{t}.
\]

\end{proposition}

When both a solution $X$ of equation (\ref{SDE}) and the Brownian motion
itself satisfy condition (\ref{condizione integ b}), we may apply a result of
absolutely continuous change of measures, see Liptser--Shiryaev \cite[theorems
7.7 and 7.9]{LS}. We know that Brownian motion satisfies this condition, when
$b\in L_{p}^{q}\left(  T\right)  $, by remark \ref{oss pq-p'q'}. We have to
impose by assumption the condition (\ref{condizione integ b}) on solutions.

\begin{corollary}
\label{prop Girsanov formulas} Take $b\in L_{p}^{q}\left(  T\right)  $ for
$p,q$ such that (\ref{pq}) holds. Let $(X,W)$ be a (weak) solution of equation
(\ref{SDE}) in the sense of theorem \ref{teo weak existence}, in particular
with $X$ satisfying condition (\ref{condizione integ b}). Then, for any non
negative Borel function $\Phi$ defined on the space $C\left(  [0,T];\mathbb{R}%
^{d}\right)  $ we have
\begin{equation}
E\left[  \Phi(X)\right]  =E\left[  \Phi(x+W)\ e^{\,\int_{0}^{T}b(s,x+W_{s}%
)\cdot\mathrm{d}W_{s}-1/2\int_{0}^{T}\left\vert b(s,x+W_{s})\right\vert
^{2}\mathrm{d}s}\right]  . \label{formula Gir 2}%
\end{equation}
In particular, weak uniqueness holds for the equation (\ref{SDE}), in the
class of solutions satisfying (\ref{condizione integ b}). Moreover, if $f\in
L_{\tilde{p}}^{\tilde{q}}(T)$ where $\tilde{p},\tilde{q}$ are such that
$d/\tilde{p}+2/\tilde{q}<1$, then, for any $k\in\mathbb{R}$ there exists a
constant $C_{f}$ depending on $\Vert f\Vert_{L_{\tilde{p}}^{\tilde{q}}(T)}$
such that
\begin{equation}
E\left[  e^{k\int_{0}^{T}\left\vert f(t,X_{t})\right\vert ^{2}\,\mathrm{d}%
t}\right]  \leq C_{f}. \label{Novikov b(X)}%
\end{equation}

\end{corollary}

\vspace{0cm}The first part of the corollary depends on the above mentioned
results of \cite[theorems 7.7 and 7.9]{LS}. To prove the exponential
integrability of $\left\vert f(t,X_{t})\right\vert ^{2}$, notice that by
(\ref{formula Gir 2}) we have%
\[
E\left[  e^{k\int_{0}^{T}\left\vert f(t,X_{t})\right\vert ^{2}\mathrm{d}%
t}\right]  =E\left[  \ e^{\,\int_{0}^{T}b(s,x+W_{s})\cdot\mathrm{d}%
W_{s}-1/2\int_{0}^{T}\left\vert b(s,x+W_{s})\right\vert ^{2}\mathrm{d}%
s+k\int_{0}^{T}\left\vert f(t,x+W_{t})\right\vert ^{2}\mathrm{d}t}\right]
\]
and thus it is sufficient to repeat the estimates made above to prove that
$E\left[  \rho_{T}^{\alpha}\right]  $ was finite.

With the same proof, namely%
\[
E\left[  \left\vert X_{t}\right\vert ^{p}\right]  =E\left[  \ \left\vert
x+W_{t}\right\vert ^{p}e^{\,\int_{0}^{T}b(s,x+W_{s})\cdot\mathrm{d}%
W_{s}-1/2\int_{0}^{T}\left\vert b(s,x+W_{s})\right\vert ^{2}\mathrm{d}%
s}\right]
\]
followed by H\"{o}lder inequality as in the proof made above to prove that
$E\left[  \rho_{T}^{\alpha}\right]  $ was finite, we also have:

\begin{proposition}
\label{proposition powers}Let $(X,W)$ be a (weak) solution of equation
(\ref{SDE}). Then
\[
\sup_{t\in\left[  0,T\right]  }E\left[  \left\vert X_{t}\right\vert
^{p}\right]  <\infty
\]
for every $p\geq1$.
\end{proposition}

\subsection*{Acknowledgements}

Part of the work was done at the Newton institute for Mathematical Sciences in Cambridge (UK), whose support is gratefully acknowledged, during the program "Stochastic partial differential equations".

\end{document}